\documentclass[11pt,a4paper]{amsart}
\usepackage{graphics,epic}
\usepackage{amsmath,amssymb, amsthm}
\usepackage[all,2cell]{xy}
\usepackage{xcolor}

\newtheorem{theorem}{Theorem}[section]
\newtheorem*{theorem*}{Theorem}

\newtheorem{lemma}[theorem]{Lemma}
\newtheorem{proposition}[theorem]{Proposition}
\newtheorem{corollary}[theorem]{Corollary}

\newtheorem*{conjecture*}{Conjecture}
\newtheorem{question}[theorem]{Question}
\newtheorem{example}[theorem]{Example}

\newcommand{\ie}{{\em i.e.}\ }

\newcommand{\ul}[1]{\underline{#1}}

\newcommand{\opname}[1]{\operatorname{\mathsf{#1}}}

\newcommand{\proj}{\opname{proj}\nolimits}

\newcommand{\Mod}{\opname{Mod}\nolimits}

\newcommand{\per}{\opname{per}\nolimits}

\newcommand{\thick}{\opname{thick}\nolimits}

\newcommand{\cok}{\opname{cok}\nolimits}

\newcommand{\Z}{\mathbb{Z}}

\newcommand{\ra}{\rightarrow}

\newcommand{\Hom}{\opname{Hom}}

\newcommand{\cHom}{{\mathcal{H}\it{om}}}
\newcommand{\cEnd}{\mathcal{E}\it{nd}}

\newcommand{\Ext}{\opname{Ext}}

\newcommand{\ten}{\otimes}
\newcommand{\lten}{\overset{\boldmath{L}}{\ten}}

\newcommand{\cd}{{\mathcal D}}

\newcommand{\ulx}{\underline{x}}
\newcommand{\uly}{\underline{y}}

\newcommand{\op}{\mathrm{op}}
\newcommand{\Cone}{\opname{Cone}\nolimits}
\newcommand{\LL}{\ell\ell}

\numberwithin{equation}{section}

\begin{document}

\title[Stratifications of algebras with two simple modules]{Stratifications of algebras with two simple modules}

\author{Qunhua Liu}
\address{Qunhua Liu, Institute of Mathematics, School of Mathematical Sciences, Nanjing Normal University, 1 Wenyuan Road, Yadong New District, Nanjing 210023, P.R.China} \email{05402@njnu.edu.cn}
\thanks{The first named author acknowledges partial support by Jiangsu Province BK20130899. Both authors acknowledge partial support by Natural Science Foundation of China 11301272. They thank a referee for helpful suggestions.}

\author{Dong Yang}
\address{Dong Yang, Department of Mathematics, Nanjing University, 22 Hankou Road, Nanjing 210093, P. R. China}
\email{dongyang2002@gmail.com}
%\thanks{Dong Yang acknowledges partial support by Natural Science Foundation of China 11301272.}
\date{Last modified on \today}

\begin{abstract} Let $A$ be a finite-dimensional algebra with two simple modules. It is shown that if the derived category of $A$ admits a stratification with simple factors being the base field $k$, then $A$ is derived equivalent to a quasi-hereditary algebra. As a consequence, if further $k$ is algebraically closed and $A$ has finite global dimension, then $A$ is either derived simple or derived equivalent to a quasi-hereditary algebra.
\end{abstract}

\maketitle

%\tableofcontents

\section{Introduction}

Quasi-hereditary algebras,  introduced by Cline, Parshall and Scott \cite{ClineParshallScott88b,Scott87}, are  finite-dimensional algebras with certain ring-theoretical stratifications, which induce  stratifications of their derived categories with simple factors being finite-dimensional division algebras, due to \cite{Cline-Parshall-Scott88}.
One asks if the converse is true, roughly speaking, if a stratification of the derived category induces a ring-theoretical stratification of the algebra. This can possibly be true only up to derived equivalence, because the class of quasi-hereditary algebras is not closed under derived equivalence, see \cite[Example]{DlabRingel89a} and Section~\ref{ss:qh-alg-rank-2} for examples.

\begin{question}
Let $k$ be a field and $A$ be a finite-dimensional $k$-algebra. Assume that $\cd(\Mod A)$ admits a stratification with simple factors being finite-dimensional division $k$-algebras. Is $A$ then derived equivalent to a quasi-hereditary algebra?
\end{question}

In this paper, we give a partial answer to the question for $A$ with two (isomorphism classes of) simple modules.

\begin{theorem}[=\ref{t:stratification-implies-qh}]\label{t:thm1} Let $k$ be a field and let $A$ be a finite-dimensional $k$-algebra with two simples modules. If $\cd(\Mod A)$ admits a stratification with all simple factors being $k$, then $A$ is derived equivalent to a quasi-hereditary algebra.
\end{theorem}

The key point is to prove that a certain differential graded algebra, which is derived equivalent to $A$, has cohomologies concentrated in two successive degrees. This is achieved by comparing Hochschild cohomologies. As a consequence of Theorem~\ref{t:thm1}, there is the following dichotomy, which was used in the introduction of \cite{Membrillo-Hernandez94} but was not proved in the literature.

\begin{theorem}[=\ref{t:dichotomy}]\label{t:thm2} Let $k$ be an algebraically closed field and
let $A$ be a finite-dimensional $k$-algebra with two simple modules. Assume that $A$ has finite global dimension. Then $A$ is either derived simple or derived equivalent to a quasi-hereditary algebra.
\end{theorem}

Theorems \ref{t:thm1} and \ref{t:thm2} are proved in Section~\ref{s:the-main-result}. In Section~\ref{s:example} we illustrate by an example how Theorem~\ref{t:thm2} is applied to show the derived simplicity of algebras. More precisely,
we construct a family of algebras $A_n(\underline{x},\underline{y})$ with two  simple modules for $n\geq 0$ and two sequences of positive integers $\underline{x},\underline{y}$, which extends the family of algebras $A_n$ studied by E. Green \cite{Green80}, Happel \cite{Happel91b} and Membrillo-Hern\'andez \cite{Membrillo-Hernandez94}. We show, by comparing the top Hochschild cohomologies,  that $A_n(\underline{x},\underline{y})$ for any $n\geq 3$ and any $\underline{x},\underline{y}$ is not derived equivalent to a quasi-hereditary algebra and hence is derived simple.

\smallskip

Throughout, let $k$ be a field. All algebras and modules are over $k$. We use right modules.

%%%%%%%%%%%%%%%%%%%%%%%%%%%%%%%%%%%%%%%%%%%%%%%%%%%%%%%%%%%%%%%%%%%%%%%%%%%%%%%%%%%%%%%%%%%%%%%%%%

\section{Preliminaries}
In this section, we recall the definitions and basic properties of derived categories of differential graded algebras, Hochschild cohomologies, recollements and stratifications. We also recall the classification of quasi-hereditary algebras with two simple modules.

\subsection{Derived category} \label{ss:derived-cat}
We follow \cite{Keller94,Keller06d}.
For a differential graded (=dg) algebra $A$, let $\cd(A)$ denote the derived category of (right) dg $A$-modules. If $A$ is an (ordinary) algebra, considered as a dg algebra concentrated in degree $0$, then a dg $A$-module is nothing else but a complex of (right) $A$-modules and $\cd(A)$ is the same as $\cd(\Mod A)$, the unbounded derived category of the abelian category $\Mod A$ of $A$-modules.

For two dg $A$-modules $M$ and $N$, define $\cHom_A(M,N)$ to be the complex whose component $\cHom_A^p(M,N)$ consists of homogenous maps of \emph{graded} $A$-modules of degree $p$ from $M$ to $N$, \ie
\begin{eqnarray*}\cHom_A^p(M,N)=\big\{f\in\prod_{q\in\mathbb{Z}}\Hom_k(M^q,N^{p+q}) \, \big |
\, f(ma)=f(m)a,\ \forall m\in M\text{ and } a\in A\big\},\end{eqnarray*} and whose 
differential is given by $d(f)=d_N f-(-1)^p f d_M$ for
$f\in\cHom_A^p(M,N)$. The complex $\cEnd_A(M)=\cHom_A(M,M)$ is a dg algebra  with multiplication being the
composition of maps. For any dg $A$-module $M$, by \cite[Theorem 3.1]{Keller94}, there is a cofibrant dg $A$-module $P$ together with a quasi-isomorphism $P\ra M$ of dg $A$-modules (called the cofibrant resolution of $M$). In this case, there is a natural isomorphism for all $p\in\mathbb{Z}$
\[\Hom_{\cd(A)}(M,N[p])\cong H^p\cHom_A (P,N).\]
We refer to \cite[Section 2.12]{KellerYang11} for the definition of cofibrant dg module and remark that if $A$ is an algebra then a bounded complex of finitely generated projective $A$-modules is a cofibrant dg $A$-module.

A dg $A$-module $M$ is \emph{compact} in $\cd(A)$ if $M$ belongs to $\per(A)=\thick(A_A)$, the smallest triangulated subcategory of $\cd(A)$ which contains $A_A$ and which is closed under taking direct summands. It is a \emph{compact generator} of $\cd(A)$ if $\thick(M)=\per(A)$. 
In the case when $A$ is an algebra, $\per(A)$ is triangle equivalent to $K^b(\proj A)$, the homotopy category of bounded complexes of finitely generated projective $A$-modules.
If $M$ is a cofibrant dg $A$-module and is a compact generator of $\cd(A)$, then by \cite[Lemma 6.1 (a)]{Keller94} the two dg algebras $\cEnd_A(M)$ and $A$ are derived equivalent, namely, there is a triangle equivalence
\[?\lten_{\cEnd_A(M)} M: \ \ \cd(\cEnd_A(M))\stackrel{\simeq}{\longrightarrow} \cd(A).\]
Further, two quasi-isomorphic dg algebras are derived equivalent.

\subsection{Hochschild cohomology}\label{ss:Hochschild-cohomologies}
We follow \cite{Keller03}.
Let $A$ be a dg algebra. Then $A$ is a dg $A$-bimodule, in other words, $A$ is a dg module over the enveloping algebra $A^e:=A^{op}\otimes_k A$. For an integer $p$, the $p$-th \emph{Hochschild cohomology} of $A$ is defined as
\[HH^p(A):=\Hom_{\cd(A^e)}(A,A[p]).\]
If $A$ is a finite-dimensional algebra over the field $k$, then $HH^p(A)=\Ext^p_{A^e}(A,A)$, which vanishes for $p<0$ and for $p>\mathrm{gl.dim} A$ (see \cite[Section 1.5]{Happel89}).

Let $A$ be an algebra and $B$ be a dg algebra which is derived equivalent to $A$. Then by \cite[Section 9]{Keller94}, there is a dg $B$-$A$-bimodule $M$ such that $?\lten_B M:\cd(B)\rightarrow\cd(A)$ is a triangle equivalence. It follows from \cite[Corollary 3.2]{Keller03} that there is an isomorphism for all $p\in\mathbb{Z}$
\[HH^p(A)\cong HH^p(B).\]

\subsection{Recollement and stratification}\label{recollement}

A \emph{recollement} of derived categories of algebras is a diagram of triangulated functors
\begin{eqnarray}
\xymatrix{\cd(\Mod B)\ar[rr]|{i_*}&&\cd(\Mod A)\ar[rr]|{j^*}\ar@/^15pt/[ll]|{i^!}\ar@/_15pt/[ll]|{i^*}&&\cd(\Mod C)\ar@/^15pt/[ll]|{j_*}\ar@/_15pt/[ll]|{j_!}}\label{eq:recollement-intro}
\end{eqnarray}
satisfying certain conditions, where $A,B,C$ are algebras. 
This notion was introduced by
 Beilinson, Bernstein and Deligne in \cite{BeilinsonBernsteinDeligne82} and can be considered as a short exact sequence of derived categories. We refer to \cite{BeilinsonBernsteinDeligne82} for the precise conditions and provide only a few properties for later use. For a finite-dimensional algebra $A$, let $r(A)$ denote the number of isomorphism classes of simple $A$-modules.

\begin{lemma}\label{l:recollement}\emph{(\cite{BeilinsonBernsteinDeligne82,AngeleriKoenigLiuYang13})} Suppose there is a recollement of the form (\ref{eq:recollement-intro}). The following hold true:
\begin{itemize}
\item[(a)] $i_*$ and $j_!$ are full embeddings and $\Hom(j_!(X),i_*(Y))=0$ for any $X\in\cd(\Mod C)$ and $Y\in\cd(\Mod B)$.
\item[(b)] $A$ has finite global dimension if and only if so do $B$ and $C$.
\item[(c)] If $A$ is finite-dimensional, then so are $B$ and $C$. In this case, $r(A)=r(B)+r(C)$.
\item[(d)]  If $A$ has finite global dimension, then $j_!(C)\oplus i_*(B)$ is a compact generator of $\cd(\Mod A)$.
\end{itemize}
\end{lemma}

The recollement (\ref{eq:recollement-intro}) is non-trivial if both $B$ and $C$ are non-trivial. An algebra $A$ is \emph{derived simple} if there are no non-trivial recollements of $\cd(\Mod A)$ of the form (\ref{eq:recollement-intro}). For example, a division algebra is derived simple.

If $A$ is not derived simple, then there is a non-trivial recollement of $\cd(\Mod A)$ of the form (\ref{eq:recollement-intro}). If further $B$ and/or $C$ are not derived simple, we can then take a non-trivial recollement of $\cd(\Mod B)$ and/or $\cd(\Mod C)$. Continuing this process with every occurring algebra unless it is derived simple, eventually we obtain a \emph{stratification} of $\cd(\Mod A)$, which is analogous to composition series. The derived simple algebras obtained in this process are the \emph{simple factors} of the stratification.

Typical examples of stratifications are given for quasi-hereditary algebras (\cite{ClineParshallScott88b,Scott87}). Let $A$ be a finite-dimensional algebra and $e\in A$ an idempotent. The ideal $AeA$ is a \emph{heredity ideal} if $eAe$ is semisimple and $AeA$ is projective as a right $A$-module. Applying \cite[Corollary 3.7]{Cline-Parshall-Scott88} and \cite[Proposition 4.1]{AngeleriKoenigLiuYang13}, we obtain a recollement
\vspace{10pt}
\[
\xymatrix{
\cd(\Mod A/AeA)\ar[rr]&&\cd(\Mod A)\ar[rr]\ar@/^15pt/[ll]\ar@/_15pt/[ll]&&\cd(\Mod eAe).\ar@/^15pt/[ll]\ar@/_15pt/[ll]
\\
}
\vspace{10pt}
\]
The algebra $A$ is \emph{quasi-hereditary} if there is a \emph{heredity chain}, \ie a finite chain of ideals
\[0=J_0\subseteq J_1\subseteq \ldots\subseteq J_m=A\]
such that $J_i/J_{i-1}$ is a heredity ideal of $A/J_{i-1}$ for all $1\leq i\leq m$. We can refine this chain such that in each step the semisimple algebra associated to the heredity ideal is simple. Then by repeating the above construction, we obtain a stratification of $\cd(\Mod A)$.

\subsection{Quasi-hereditary algebras with two simple modules}\label{ss:qh-alg-rank-2}

Let  $x,y\geq 0$ be integers. Let $Q(x,y)$ be the quiver which has two vertices $1$ and $2$ and which has $x$ arrows $\alpha_1,\ldots,\alpha_x$ from $1$ to $2$ and $y$ arrows $\beta_1,\ldots,\beta_y$ from $2$ to $1$. Let $B(x,y)$ be the quotient of the path algebra of $Q(x,y)$ by the ideal generated by the paths  $\beta_q\alpha_p$ for $1\leq p\leq x$ and  $1\leq q\leq y$.
It is easy to check that  the global dimension of $B(x,y)$ is at most $2$ and that $B(x,y)$ is quasi-hereditary with heredity chain
\[0\subseteq B(x,y)e_1 B(x,y) \subseteq B(x,y).\]
Conversely, any quasi-hereditary algebra $A$ with two simple modules whose endomorphism algebras are $k$ is Morita equivalent to some $B(x,y)$, see \cite[Theorem 3.1]{Membrillo-Hernandez94}.

We remark that $B(x,y)$ and $B(y,x)$ are Ringel dual to each other, \ie $B(y,x)$ is the endomorphism algebra of the characteristic tilting module over $B(x,y)$ (see \cite{Ringel91} for the definition of the characteristic tilting module). The proof is direct and is left to the reader as an exercise. A consequence of this fact is that $B(x,y)$ and $B(y,x)$ are derived equivalent.

Next we show that, when $x\geq 2$ and $y\geq 1$, the quasi-hereditary algebra $B(x,y)$ is derived equivalent to an algebra which is not quasi-hereditary. So the class of quasi-hereditary algebras with two simples is not closed under derived equivalence.

Let $P_1$ and $P_2$ be the indecomposable projective $B(x,y)$-module at vertices $1$ and $2$, respectively.
Consider the map
\[f=(\alpha_1,\ldots,\alpha_x)^{\mathrm{tr}}\ :\ P_1\longrightarrow P_2^{\oplus x}.\]
Let $T=\cok(f)\oplus P_2$. It is straightforward to check that $T$ is a tilting module over $B(x,y)$.  (Alternatively, $T$ is the left silting mutation of $P_1\oplus P_2$ at $P_1$ in the sense of Aihara--Iyama \cite{AiharaIyama12}. It follows that $T$ is self-orthogonal and hence is a tilting module.) The endomorphism algebra $E$ of $T$ is the quotient of the path algebra of the quiver which has two vertices $1$ and $2$ and which has $x^2y$ arrows $a^{pq}_r$ ($1\leq p,q\leq x$, $1\leq r\leq y$) from $1$ to $2$ and $x$ arrows $b_p$ ($1\leq p\leq x$) from $2$ to $1$ modulo the ideal generated by the relations
$\sum_{q=1}^x b_q a^{pq}_r$ ($1\leq p\leq x$, $1\leq r\leq y$), $a^{pq}_r b_p=a^{p'q}_r b_{p'}$ ($1\leq p,q,p'\leq x$, $1\leq r\leq y$) and $a^{pq}_r b_{p'}$ ($1\leq p,q,p'\leq x$ with $p\neq p'$, $1\leq r\leq y$). It is readily seen that $E$ is not isomorphic to any $B(x',y')$, and hence is not quasi-hereditary.

\section{The main results}\label{s:the-main-result}

In this section, we prove Theorems~\ref{t:thm1}  and \ref{t:thm2}. We show that a finite-dimensional algebra, which has two simples and whose derived category admits a stratification with simple factors being the base field $k$, is derived equivalent to a certain graded $n$-Kronecker algebra, which is in turn derived equivalent to a quasi-hereditary algebra. The grading of this $n$-Kronecker algebra is determined by comparing Hochschild cohomologies. We start with computing the Hochschild cohomologies of graded $n$-Kronecker algebras.

\subsection{The graded n-Kronecker algebra: Hochschild cohomologies}

Let $n\geq 1$ be a positive integer. Let $Q$ be a graded $n$-Kronecker quiver, namely, the vertex set $Q_0$ consists of two vertices $1$ and $2$, the arrow set $Q_1$ consists of $n$ arrows, which are all from $1$ to $2$, and to each arrow we assign an integer. The corresponding graded path algebra $\Lambda=kQ$ is called a graded $n$-Kronecker algebra. We write $V=\bigoplus_{p\in\mathbb{Z}} V^p$ for the graded vector space of arrows of $Q$. Thus $V^p$ has a basis formed by arrows of $Q$ of degree $p$. Then $\Lambda=kQ_0\oplus V$. Written as a  matrix algebra, $\Lambda$ has the form
$$\Lambda=\begin{pmatrix}  k & 0 \\ V & k \end{pmatrix}.$$

Our aim in this subsection is to calculate the dimension of the Hochschild cohomologies $HH^p(\Lambda)$ of $\Lambda$, viewed as a dg algebra with trivial differential. More precisely, we show the following result.

\begin{proposition} \label{Hochschild-Kronecker} Let $Q$ be a graded $n$-Kronecker quiver and $\Lambda=kQ_0\oplus V$ its graded path algebra. Then
$$\dim_k HH^p(\Lambda) = \begin{cases} \#\{(\alpha,\beta)\in Q_1\times Q_1~|~ \deg\alpha-\deg\beta=0\}-1 & \text{ if }p=1\\
\#\{(\alpha,\beta)\in Q_1\times Q_1~|~ \deg\alpha-\deg\beta=-1\}+1 & \text{ if } p=0\\
\#\{(\alpha,\beta)\in Q_1\times Q_1~|~ \deg\alpha-\deg\beta=p-1\} & \text{ if } p\neq 0,1.\end{cases}$$
\end{proposition}

We give two corollaries before presenting the proof.

\begin{corollary} Let $\Lambda$ and $V$ be the same as in Proposition \ref{Hochschild-Kronecker}.
\begin{enumerate} \label{corollary1-Kronecker}
\item  The equality $\dim_k HH^p(\Lambda)=\dim_k HH^{2-p} (\Lambda)$ holds for all $p\neq 0,1$. And $\dim_k HH^0(\Lambda)=\dim_k HH^2(\Lambda)+1$.

\item $\sum_{p\in\Z} \dim_k HH^p(\Lambda)=(\dim_kV)^2=n^2$.

\item Let $a$ and $b$ be the minimal and the maximal degrees of arrows of $Q$, respectively. Then $HH^p(\Lambda)= 0$ unless $p\in [1-b+a,1+b-a]\cup\{0\}$. Moreover,  $HH^{1-b+a}(\Lambda)\neq 0$ if $1-b+a\leq 0$.
\end{enumerate}
\end{corollary}
\begin{proof}
These follow straightforward from Proposition~\ref{Hochschild-Kronecker}.
\end{proof}

\begin{corollary} Let $\Lambda$ and $V$ be the same as in Proposition \ref{Hochschild-Kronecker}.
If $\Lambda$ is derived equivalent to an algebra, then $V$ is concentrated in at most two successive degrees. \label{corollary2-Kronecker}
\end{corollary}

\begin{proof}
If $\Lambda$ is derived equivalent to an algebra, then $HH^p(\Lambda)$ vanishes for all $p<0$.
By Corollary \ref{corollary1-Kronecker} (3),  the latter condition implies $1-b+a\geq 0$, namely $b-a\leq 1$, which, by the definition of $a$ and $b$, means exactly that $V$ is concentrated in at most two successive degrees.
\end{proof}

Now we provide a proof of Proposition \ref{Hochschild-Kronecker}, followed by two examples as illustration.

\begin{proof}[Proof of Proposition \ref{Hochschild-Kronecker}]
Recall that $\Lambda^e$ stands for the enveloping algebra $\Lambda^\op\otimes_k \Lambda$ and the $p$-th Hochschild cohomology $HH^p(\Lambda)$ is defined to be $\Hom_{\cd(\Lambda^e)}(\Lambda,\Lambda[p])$. We stress that here $\Lambda$ and $\Lambda^e$ are viewed as dg algebras with trivial differential and $\cd(\Lambda^e)$ is the derived category of dg $\Lambda^e$-modules.

First we construct a cofibrant resolution $P$ of $\Lambda$ as dg $\Lambda^e$-module by using the Koszul bimodule resolution of $\Lambda$. We claim that the following complex is exact and is a projective resolution of $\Lambda$ as graded $\Lambda^e$-module:
\[
\xymatrix@R=0.5pc{0\ar[r] & \Lambda\otimes_{kQ_0}V\otimes_{kQ_0}\Lambda \ar[r]^(0.57){f} & \Lambda\otimes_{kQ_0}\Lambda \ar[r]^(0.65){mult.} & \Lambda \ar[r] & 0.\\
& 1\otimes v\otimes 1 \ar@{|->}[r] & v\otimes 1-1\otimes v}
\]
Indeed, by \cite[Section 3]{VandenBergh94}, this is a projective resolution of $\Lambda$ as $\Lambda^e$-module if we forget the grading. Therefore by definition $P:=\Cone(f)$, where $f$ is viewed as a homomorphism of dg $\Lambda^e$-modules, is cofibrant over $\Lambda^e$ and hence is a cofibrant resolution of $\Lambda$.

It follows that
\[
HH^p(\Lambda)\cong \Hom_{\cd(\Lambda^e)}(\Lambda,\Lambda[p]) \cong H^p{\cHom}_{\Lambda^e}(P,\Lambda).
\] Now $P$ is the mapping cone of $f$, hence
\[
{\cHom}_{\Lambda^e}(P,\Lambda)\cong \cHom_{\Lambda^e}(\Cone(f),\Lambda)\cong \Cone(\cHom_{\Lambda^e}(f,\Lambda))[-1].
\]
Further observe the isomorphisms $\Lambda\otimes_{kQ_0}V\otimes_{kQ_0}\Lambda \cong V\otimes_{kQ_0^e}\Lambda^e$ and $\Lambda\otimes_{kQ_0}\Lambda\cong\Lambda\otimes_{kQ_0}kQ_0\otimes_{kQ_0}\Lambda \cong kQ_0\otimes_{kQ_0^e}\Lambda^e$. So the map $f^*:=\cHom_{\Lambda^e}(f,\Lambda)$ can be calculated explicitly:
\[
\xymatrix{
\cHom_{\Lambda^e}(\Lambda\otimes_{kQ_0}\Lambda,\Lambda) \ar[r]^(0.45){f^*}\ar[d]^{\cong} & \cHom_{\Lambda^e}(\Lambda\otimes_{kQ_0}V\otimes_{kQ_0}\Lambda,\Lambda)\ar[d]^{\cong} \\
\cHom_{\Lambda^e}(kQ_0\otimes_{kQ_0^e}\Lambda^e,\Lambda) \ar[r]\ar[d]^\cong & \cHom_{\Lambda^e}(V\otimes_{kQ_0^e}\Lambda^e,\Lambda)\ar[d]^\cong \\
\cHom_{kQ_0^e}(kQ_0,\Lambda) \ar[r]\ar[d]^\cong & \cHom_{kQ_0^e}(V,\Lambda)\ar[d]^\cong \\
e_1\Lambda e_1\oplus e_2\Lambda e_2  \ar[r]^(0.47)g & \bigoplus_{\rho_\in Q_1} e_2\Lambda e_1[\deg\rho]
}
\]
where $e_1$ and $e_2$ are the two primitive idempotents of $\Lambda$. The middle vertical isomorphisms are obtained by adjunction. The lower vertical isomorphisms are obtained by observing the two isomorphisms of graded $kQ_0^e$-modules: $V=e_2\Lambda e_1\cong \bigoplus_{\rho\in Q_1}((e_2\ten e_1)kQ_0^e)[-\deg\rho]$ and $kQ_0=(e_1\ten e_1)kQ_0^e\oplus (e_2\ten e_2)kQ_0^e$. We have the formula
\[
\dim_k (e_2\Lambda e_1[\deg\rho])^p=\#\{\alpha\in Q_1~|~\deg\alpha-\deg\rho=p\}.
\]
So
\[
\dim_k (\bigoplus_{\rho\in Q_1}e_2\Lambda e_1[\deg\rho])^p=\#\{(\alpha,\beta)\in Q_1\times Q_1~|~\deg\alpha-\deg\beta=p\}.
\]
Finally, the map $g$ in the last row is given by $g(e_1)=(\rho[\deg\rho])_{\rho\in Q_1}=-g(e_2)$. In particular, both the image and the kernel of $g$ are one-dimensional. Since both the domain and the codomain of $g$ have trivial differentials, the desired formula follows.
%To obtain the isomorphism in the right lower corner, note that $V=e_2\Lambda e_1$ and a $kQ_0^e$-homomorphism from $V$ to $\Lambda$ assigns each arrow $\alpha\in Q_1$ some arrow $\beta\in Q_1$, and the degree of this homomorphism is determined by the difference of the degrees of $\beta$ and $\alpha$. Thus,
%$$\dim_k\cHom_{kQ_0^e}^p(V,\Lambda) = \#\{(\alpha,\beta)\in Q_1\otimes Q_1: \deg\beta-\deg\alpha=p\}.$$ Notice that the
% complex $\cHom_{kQ_0^e}(V,\Lambda)$ has trivial differential, as so do $V$ and $\Lambda$.
%It follows that
%$$\dim_k HH^p(\Lambda)=\begin{cases}
%\dim_k\cHom_{kQ_0^e}^{0}(V,\Lambda)-1 & p=1 \\
%\dim_k\cHom_{kQ_0^e}^{-1}(V,\Lambda)+1 & p=0 \\
%\dim_k\cHom_{kQ_0^e}^{i-1}(V,\Lambda) & p\neq 0,1 . \end{cases}$$
%Combining the information we have just collected yields the desired formula.
\end{proof}

\begin{example}
Suppose that we have a graded $3$-Kronecker quiver with $Q_1=\{\alpha_{-1},\alpha_0,\alpha_2\}$ and $\deg\alpha_p=p$ for $p=-1,0,2$. Then $V=V^{-1}\oplus V^0\oplus V^2$, where each $V^p$ is one-dimensional with basis $\alpha_p$.
%The graded $3$-Kronecker algebras $\Lambda$ is the direct sum of $V$ and $kQ_0\cong k^2$.
In this case we have $\cHom_{kQ_0^e}(V,\Lambda)\cong V[-1]\oplus V\oplus V[2]$.% and $$\dim_k\cHom^p_{kQ_0^e}(V,\Lambda)=\begin{cases} 1 & p=\pm3,\pm2,\pm1 \\
%3 & p=0 \\
%0 & {\text otherwise}. \end{cases} $$
It follows that $$\dim_k HH^p(\Lambda)=\begin{cases} 1 & p=-2,-1,2,3,4 \\
2 & p=0 \\
2 & p=1\\
0 & {\text otherwise}. \end{cases}$$
\end{example}

\begin{example}\label{ex:hochschild-qh} Let $x,y\geq 0$ be integers such that $x+y\geq 1$.
Suppose that we have a graded $n$-Kronecker quiver with $n=x+y$, $Q_1=\{\alpha_{0,1},\ldots,\alpha_{0,y},\alpha_{1,1},\ldots,\alpha_{1,x}\}$ and $\deg\alpha_{p,q}=p$. Then $V=V^{0}\oplus V^1$, where $V^0$ is $y$-dimensional with basis $\{\alpha_{0,1},\ldots,\alpha_{0,y}\}$, and $V^1$ is $x$-dimensional with basis $\{\alpha_{1,1},\ldots,\alpha_{1,x}\}$.
%The corresponding graded $(x+y)$-Kronecker algebra $\Lambda$ has dimension $x+y+2$.
Here $\cHom_{kQ_0^e}(V,\Lambda)\cong V^{\oplus y}\oplus V[1]^{\oplus x}$. It follows that %$$\dim_k\cHom^p_{kQ_0^e}(V,\Lambda)=\begin{cases} xy & p=\pm1 \\
%x^2+y^2 & p=0 \\
%0 & {\text otherwise} \end{cases}$$
%and 
$$\dim_k HH^p(\Lambda)=\begin{cases} xy+1 & p=0 \\
x^2+y^2-1 & p=1 \\
xy & p=2\\
0 & {\text otherwise}.\end{cases} $$
\end{example}

\subsection{The main results}

\begin{theorem}\label{t:stratification-implies-qh} Let $A$ be a finite-dimensional $k$-algebra with two simple modules. Assume that $\cd(\Mod A)$ admits a stratification with all simple factors being $k$. Then $A$ is derived equivalent to a quasi-hereditary algebra.
\end{theorem}

\begin{proof} First we show that  such an algebra $A$ must be derived equivalent to a graded $n$-Kronecker algebra $\Lambda$.
By Lemma \ref{l:recollement} (c), a stratification of $\cd(\Mod A)$ has precisely two simple factors, namely, it is a recollement of the form
\begin{eqnarray}
\xymatrix{\cd(\Mod k)\ar[rr]|{i_*}&&\cd(\Mod A)\ar[rr]|{j^*}\ar@/^15pt/[ll]|{i^!}\ar@/_15pt/[ll]|{i^*}&&\cd(\Mod k).\ar@/^15pt/[ll]|{j_*}\ar@/_15pt/[ll]|{j_!}}\label{eq:recollement-in-main-thm}
\end{eqnarray}
Take $T=i_*(B)\oplus j_!(C)\in\cd(\Mod A)$. By Lemma \ref{l:recollement} (d), $T$ is a compact generator of $\cd(\Mod A)$. We may assume that $T$ is a bounded complex of finitely generated projective $A$-modules, in particular, it is a cofibrant dg $A$-module. Then $A$ is derived equivalent to the dg endomorphism algebra $\tilde{\Lambda}=\cEnd_A(T)$. Written as a matrix algebra, $\tilde{\Lambda}$ has the form
\[
\tilde{\Lambda}=\begin{pmatrix} \cEnd_A(i_*(k)) & \cHom_A(j_!(k),i_*(k)) \\ \cHom_A(i_*(k),j_!(k)) & \cEnd_A(j_!(k)) \end {pmatrix}.
\]
Now by Lemma \ref{l:recollement} (a), the space $\Hom_{\cd(\Mod A)}(i_*(k),i_*(k)[p])$ vanishes if $p\neq 0$ and is $k$ if $p=0$. Therefore, since $i_*(k)$ is cofibrant over $A$, the dg algebra $\cEnd_A(i_*(k))$ is quasi-isomorphic to $k$. It follows that the embedding $k\cong k\{\mathrm{id}_{i_*(k)}\}\hookrightarrow \cEnd_A(i_*(k))$ is a quasi-isomorphism of dg algebras. For the same reason,
 the embedding $k\cong k\{\mathrm{id}_{j_!(k)}\}\hookrightarrow\cEnd_A(j_!(k))$ is a quasi-isomorphism of dg algebras, and similarly, the complex $\cHom_A(j_!(k),i_*(k))$ is acyclic. Further, we can take a subcomplex $V$ of $\cHom_A(i_*(k),j_!(k))$ such that the embedding $V\hookrightarrow \cHom_A(i_*(k),j_!(k))$ is a quasi-isomorphism of complexes of vector spaces (thus as a graded vector space $V$ is isomorphic to $\bigoplus_{p\in\mathbb{Z}}\Hom_{\cd(\Mod A)}(i_*(k),j_!(k)[p])$). It follows that the embedding
\[
\Lambda=\left(\begin{array}{cc} k & 0 \\ V & k \end{array}\right)
\hookrightarrow
\tilde{\Lambda}
\]
%\[
%\xymatrix{
%\Lambda=\left(\begin{array}{cc} k & 0 \\ V & k \end{array}\right)\cong\left(\begin{array}{cc} k\{\mathrm{id}_{i_*(k)}\} & 0 \\ V & k\{\mathrm{id}_{j_!(k)}\} \end{array}\right)
%\hookrightarrow
%\ar@{^{(}->}[r] &
%\left(\begin{array}{cc} \cEnd_A(i_*(k)) & \cHom_A(j_!(k),i_*(k)) \\ \cHom_A(i_*(k),j_!(k)) & \cEnd_A(j_!(k)) \end{array}\right)=
%\tilde{\Lambda}
%}
%\]
is a quasi-isomorphism of dg algebras.
Therefore, $\Lambda$, viewed as a dg algebra with trivial differential, is derived equivalent to $\tilde{\Lambda}$, and hence it is derived equivalent to $A$.

Now by Corollary \ref{corollary2-Kronecker}  the graded vector space $V$ is concentrated in at most two successive degrees. We can assume that $V$ is concentrated in degrees $0$ and $1$, \ie $V=V^0\oplus V^1$, by shifting $i_*(B)$ in the definition of $T$ if necessary. 

Next we consider the quasi-hereditary algebra $B=B(x,y)$ for $x=\dim_k V^1$ and $y=\dim_k V^0$ (see Section \ref{ss:qh-alg-rank-2}). The derived category $\cd(\Mod B)$ admits a recollement of the form (\ref{eq:recollement-in-main-thm}),
and here $i_*(k)=B/Be_1B$ and $j_!(k)=e_1B$. Direct computation shows that $\Ext^p_B(B/Be_1B,e_1B)\neq 0$ only for $p=0$ and $1$; Moreover, $\dim_k\Hom_B(B/Be_1B,e_1B)=y$ and $\dim_k\Ext^1_B(B/Be_1B,e_1B)=x$. Therefore, we have isomorphisms of graded vector spaces 
\[
\bigoplus_{p\in\mathbb{Z}}\Hom_{\cd(\Mod B)}(B/Be_1B,e_1B[p])\cong\bigoplus_{p\in\mathbb{Z}}\Ext^p_B(B/Be_1B,e_1B)\cong k^{\oplus y}\oplus k^{\oplus x}[-1].\] Applying the above argument, one sees that $B$ is derived equivalent to the graded $n$-Kronecker algebra
$$\begin{pmatrix} k & 0 \\
k^{\oplus y}\oplus k^{\oplus x}[-1] & k \end{pmatrix}.$$
This is precisely the algebra $\Lambda$ we have constructed above.

To summarise, the algebra $A$ is derived equivalent to a graded $n$-Kronecker algebra, which is in turn derived equivalent to a quasi-hereditary algebra $B$. We are done.
\end{proof}

Next we prove the following dichotomy.

\begin{theorem}\label{t:dichotomy}
Let $k$ be algebraically closed and $A$ be a finite-dimensional  $k$-algebra with two simple modules. Assume that $A$ has finite global dimension. Then $A$ is either derived simple or derived equivalent to a quasi-hereditary algebra.
\end{theorem}
\begin{proof}
Suppose that $A$ is not derived simple. Then the derived category $\cd(\Mod A)$ admits a  recollement
\begin{eqnarray*}
\xymatrix{\cd(\Mod B)\ar[rr]|{i_*}&&\cd(\Mod A)\ar[rr]|{j^*}\ar@/^15pt/[ll]|{i^!}\ar@/_15pt/[ll]|{i^*}&&\cd(\Mod C),\ar@/^15pt/[ll]|{j_*}\ar@/_15pt/[ll]|{j_!}}\end{eqnarray*}
where $B$ and $C$ are non-trivial algebras. It follows from Lemma~\ref{l:recollement} that $B$ and $C$ are finite-dimensional local algebras of finite global dimension. Thus $B$ and $C$ are Morita equivalent to $k$, since the base field $k$ is algebraically closed. Without loss of generality we may assume that $B$ and $C$ are the field $k$. As a consequence, the above recollement is a stratification of $\cd(\Mod A)$ with simple factors being $k$. By Theorem \ref{t:stratification-implies-qh}, $A$ is derived equivalent to a quasi-hereditary algebra.
\end{proof}

Some corollaries are given below.

For $x,y\geq 0$, consider  the graded $(x+y)$-Kronecker quiver with $x+y$ arrows from $1$ to $2$, $x$ of which are in degree $0$ and $y$ of which are in degree $1$. Let $\Lambda(x,y)$ denote the graded path algebra of this quiver. If $x=y=0$, then $\Lambda(0,0)=B(0,0)$ is the product of two copies of the field $k$. It is known that $HH^0(B(0,0))=k\times k$ and $HH^p(B(0,0))$ vanishes for $p\geq 1$. In the proof of Theorem~\ref{t:stratification-implies-qh} we have shown that the quasi-hereditary algebra $B(x,y)$ with $x+y\geq 1$ is derived equivalent to $\Lambda(y,x)$. Thus we have by Example~\ref{ex:hochschild-qh}

\begin{corollary}\label{c:hochschild-qh}  
The dimensions of the Hochschild cohomologies of $B(x,y)$ is given by
\begin{eqnarray*}  \dim_k HH^p(B(x,y))&=&\begin{cases} xy+1 & p=0 \\
x^2+y^2-1 & p=1 \\
xy & p=2\\
0 & {\text otherwise}\end{cases}\qquad \text{if $x+y\geq 1$,}\\
\dim_k HH^p(B(0,0))&=&\begin{cases} 2 & p=0\\
0 & {\text otherwise}.
\end{cases}
\end{eqnarray*}
\end{corollary}

There is the following sufficient condition for derived simplicity for finite-dimensional algebras with two simple modules and of finite global dimension.

\begin{corollary}\label{c:derived-simplicity-from-hoch-coh}
Let $k$ be algebraically closed and $A$ a finite-dimensional $k$-algebra with two simple modules. Assume that $A$ has finite global dimension. If $HH^p(A)\neq 0$ for some $p>2$, then $A$ is derived simple.
\end{corollary}
\begin{proof}
Suppose that $A$ is not derived simple. Then by Theorem~\ref{t:dichotomy}, $A$ is derived equivalent to some quasi-hereditary algebra with two simples, \ie to some $B(x,y)$. By Corollary~\ref{c:hochschild-qh}, the space $HH^p(A)$ vanishes for all $p>2$.
\end{proof}

\begin{corollary}
There are natural one-to-one correspondences among the following sets:
\begin{itemize}
\item[(i)] Morita equivalence classes of quasi-hereditary algebras with two simple modules whose endomorphism algebras are isomorphic to $k$, modulo Ringel duality;
\item[(ii)] $B(x,y)$ for $y\geq x\geq 0$;
\item[(iii)] graded $n$-Kronecker algebras $\Lambda(x,y)$ for $x\geq y\geq 0$;
\item[(iv)] derived equivalence classes of finite-dimensional algebras with two simple modules such that the derived category admits a stratification with all simple factors being $k$.
\end{itemize}
\end{corollary}
\begin{proof}
{\rm [(ii)$\leftrightarrow$(iii)]} The map $B(x,y)\mapsto \Lambda(y,x)$ is one-to-one and onto.

{\rm [(ii)$\leftrightarrow$(iv)]} It follows from Theorem~\ref{t:stratification-implies-qh} that in each derived equivalence class in (iv) % of finite-dimensional algebras with two simple modules such that the derived category admits a stratification with all simple factors being $k$ 
there is an algebra of the form $B(x,y)$ for some $x,y\geq 0$. Since $B(x,y)$ and $B(y,x)$ are derived equivalent, we can choose $x$ and $y$ such that $y\geq x$. It remains to show that $B(x',y')$ is derived equivalent to $B(x,y)$ if and only if $(x',y')=(x,y)$ or $(x',y')=(y,x)$. Suppose that $B(x,y)$ and $B(x',y')$ are derived equivalent. Then they have isomorphic Hochschild cohomologies. Thus by Corollary~\ref{c:hochschild-qh} we have
\[
\begin{cases}
xy+1=x'y'+1\\
x^2+y^2-1=x'^2+y'^2-1\\
xy=x'y'
\end{cases}  %\text{ or\ \ \  } 
%\begin{cases} 
%xy+1=2\\
%x^2+y^2-1=0\\
%xy=0
%\end{cases} \text{ or\ \ \ }
%\begin{cases}
%x'y'+1=2\\
%x'^2+y'^2-1=0\\
%x'y'=0
%\end{cases} 
\text{ or\ \ \ } x=y=0=x'=y'.
\]
Solving this equation we obtain $(x',y')=(x,y)$ or $(x',y')=(y,x)$.

{\rm [(i)$\leftrightarrow$(ii)]} First, in each Morita equivalence class in (i) %of quasi-hereditary algebras with two simple modules whose endomorphism algebras are $k$ 
there is an algebra of the form $B(x,y)$ for some $x,y\geq 0$. Secondly, as remarked in Section~\ref{ss:qh-alg-rank-2}, the algebras $B(x,y)$ and $B(y,x)$ are Ringel dual to each other. Finally, if $y\geq x\geq 0$, $y'\geq x'\geq 0$ and $(x,y)\neq (x',y')$, then $B(x,y)$ and $B(x',y')$ are not derived equivalent, and hence they are not Morita equivalent.
\end{proof}

%%%%%%%%%%%%%%%%%%%%%%%%%%%%%%%%%%%%%%%%%%%%%%%%%%%%%%%%%%%%%%%%%%%%%%%%%%%%%%%%%%%%%%%%%%%%%%%%%%%%%%%%%%%%%%

\section{An example}\label{s:example}
Let $k$ be algebraically closed.
In this section, we construct a family $A_n(\underline{x},\underline{y})$ ($n\geq 0$, $\underline{x},\underline{y}$ are sequences of positive integers) of algebras with two simples. We will show that if $n\leq 2$ then the algebra $A_n(\underline{x},\underline{y})$ is isomorphic to a quasi-hereditary algebra; otherwise, it is derived simple.

When $\underline{x}=\underline{y}=(1,1,\ldots)$, the algebra $A_n(\underline{x},\underline{y})$ is the same as the Fibonacci algebra $A_n$ constructed by E. Green \cite{Green80} and further studied by Happel \cite{Happel91b} and Membrillo-Hern\'andez \cite{Membrillo-Hernandez94}. In \cite{Happel91b,Membrillo-Hernandez94}, it is shown that $A_n$ ($n\geq 3$) is not derived equivalent to quasi-hereditary algebras. The method used in \cite{Membrillo-Hernandez94} works for our general $A_n(\ulx,\uly)$ with $n$ even but not for $n$ odd. In this section we provide a different method, which uses the top Hochschild cohomology.

\subsection{The construction}

Let $\ul{x}=(x_1,x_2,\ldots),\,\ul{y}=(y_1,y_2,\ldots)$ be two sequences of positive integers. We define a family of algebras $A_n(\ulx,\uly)$ recursively as follows.

$A_0(\ulx,\uly)$ is the path algebra of the quiver with two vertices $1$ and $2$ and without arrows. Namely, $A_0(\ulx,\uly)=k\oplus k$.

$A_1(\ulx,\uly)$ is the path algebra of the $x_1$-Kronecker quiver: it has two vertices $1$ and $2$, and has  $x_1$ arrows, all of which go from $1$ to $2$.

For $m\geq 1$, the algebra $A_{2m}(\ulx,\uly)$ is
obtained from $A_{2m-1}(\underline{x},\underline{y})$ by adding $y_m$ new arrows $\beta_{m,j}$ ($1\leq j\leq y_m$) from $2$ to $1$ and adding as new relations the paths $\beta_{m,j}\alpha_{i',j'}$ ($1\leq j\leq y_m$, $1\leq i'\leq m$ and $1\leq j'\leq x_{i'}$); the algebra $A_{2m+1}(\ulx,\uly)$ is obtained from $A_{2m}(\ulx,\uly)$ by adding $x_{m+1}$ new arrows $\alpha_{m+1,j}$ ($1\leq j\leq x_{m+1}$) from $1$ and $2$ and adding as new relations the paths $\alpha_{m+1,j}\beta_{i',j'}$ ($1\leq j\leq x_{m+1}$, $1\leq i'\leq m$ and $1\leq j'\leq y_{i'}$).

In other words, the quiver of $A_{n}(\ulx,\uly)$ has two vertices $1$ and $2$ and has $x_1+\ldots+x_{\lfloor \frac{n+1}{2}\rfloor}$ arrows from $1$ to $2$ labelled by $\alpha_{i,j}$ ($1\leq i\leq \lfloor\frac{n+1}{2}\rfloor$ and $1\leq j\leq x_i$) and $y_1+\ldots+y_{\lfloor\frac{n}{2}\rfloor}$ arrows labelled by $\beta_{i,j}$ ($1\leq i\leq \lfloor\frac{n}{2}\rfloor$ and $1\leq j\leq y_i$). Here for a real number $r$ we denote by $\lfloor r\rfloor$ the greatest integer smaller than or equal to $r$. The relations of $A_n(\ulx,\uly)$ are $\alpha_{i,j}\beta_{i',j'}$ ($1\leq i'<i\leq \lfloor\frac{n+1}{2}\rfloor$, $1\leq j\leq x_i$ and $1\leq j'\leq y_{i'}$) and $\beta_{i',j'}\alpha_{i,j}$ ($1\leq i\leq i'\leq \lfloor\frac{n}{2}\rfloor$, $1\leq j'\leq y_{i'}$ and $1\leq j\leq x_i$). %These relations ensure that the longest paths in $A_{2m}(\ulx,\uly)$ are $\alpha_{1,j_1}\beta_{1,j'_1}\alpha_{2,j_2}\beta_{2,j'_2}\cdots\alpha_{m,j_m}\beta_{m,j'_m}$ ($1\leq j_i\leq x_i$ and $1\leq j'_i\leq y_i$ for $1\leq i\leq m$), and the longest paths in $A_{2m+1}(\ulx,\uly)$ are $\alpha_{1,j_1}\beta_{1,j'_1}\alpha_{2,j_2}\beta_{2,j'_2}\cdots\alpha_{m,j_m}\beta_{m,j'_m}\alpha_{m+1,j_{m+1}}$ ($1\leq j_i\leq x_i$ for $1\leq i\leq m+1$ and $1\leq j'_i\leq y_i$ for $1\leq i\leq m$).

Notice that $A_1(\ulx,\uly)=B(x_1,0)\cong B(0,x_1)$ and $A_2(\underline{x},\underline{y})=B(x_1,y_1)$. Thus (up to Morita equivalence) the set $\{A_n(\ulx,\uly)~|~n\leq 2,~\ulx,\uly \text{ arbitrary}\}$ consists of all quasi-hereditary algebras with two simple modules.

\subsection{The Cartan matrix}\label{ss:cartan-matrix}

Let $e_1,e_2$ be the two primitive idempotents of the algebra $A_n(\ulx,\uly)$. Recall that the Cartan matrix $C_n(\ulx,\uly)=(c_{ij})$ of $A_n(\ulx,\uly)$ is the $2\times 2$ matrix with entries $c_{ij}=\dim_k(e_iAe_j)$, which count precisely the paths starting from $j$ and ending at $i$. One can check directly that
\begin{eqnarray*}
C_0(\ulx,\uly)&=&\begin{pmatrix} 1 & 0 \\ 0 & 1 \end{pmatrix},\\
C_1(\ulx,\uly)&=& \begin{pmatrix} 1 &  0\\x_1 & 1 \end{pmatrix},\\
C_2(\ulx,\uly)&=&\begin{pmatrix} 1 & y_1 \\ x_1 &1+x_1y_1 \end{pmatrix}, \\
C_3(\ulx,\uly)&=&\begin{pmatrix} 1+y_1x_2 & y_1\\ x_1+x_2+x_1y_1x_2  & 1+x_1y_1 \end{pmatrix}, \\
C_4(\ulx,\uly)&=&\begin{pmatrix} 1+y_1x_2 &  y_1+y_2+y_1x_2y_2 \\ x_1+x_2+x_1y_1x_2 & 1+x_1y_1+x_1y_2+x_2y_2+x_1y_1x_2y_2 \end{pmatrix}.
\end{eqnarray*}

We will use the following `generalized Fibonacci numbers'
\[
F(\ulx,\uly)=(F_0(\ulx,\uly),F_1(\ulx,\uly),F_2(\ulx,\uly),\ldots)
\]
to describe $C_n(\ulx,\uly)$: Here $F_0(\ulx,\uly)=0,~F_1(\ulx,\uly)=1$ and inductively
\begin{eqnarray*}
F_{2m}(\ulx,\uly)&=& F_{2m-2}(\ulx,\uly)+F_{2m-1}(\ulx,\uly)x_m,\\
F_{2m+1}(\ulx,\uly) &=& F_{2m-1}(\ulx,\uly)+F_{2m}(\ulx,\uly)y_m.
\end{eqnarray*}
Indeed, $F_2(\ulx,\uly)=x_1,~F_3(\ulx,\uly)=1+x_1y_1,~F_4(\ulx,\uly)=x_1+x_2+x_1y_1x_2,~F_5(\ulx,\uly)=1+x_1y_1+x_1y_2+x_2y_2+x_1y_1x_2y_2$.
When $\ulx=\uly=(1,1,\ldots)$, the generalised Fibonacci numbers are exactly the usual Fibonacci numbers.

The following lemma can be proved inductively. For a sequence $\ulx$, let $S\ulx=(x_2,x_3,\ldots)$ denote its shift.

\begin{lemma}\label{l:cartan} The following formulas hold
\begin{enumerate}
\item $C_{2m}(\ulx,\uly)= \begin{pmatrix} 1 & y_m \\ 0 & 1 \end{pmatrix}\times C_{2m-1}(\ulx,\uly) =
\begin{pmatrix} F_{2m-1}(\uly,S\ulx) & F_{2m}(\uly,S\ulx) \\ F_{2m}(\ulx,\uly) & F_{2m+1}(\ulx,\uly) \end{pmatrix}$.

\item $C_{2m+1}(\ulx,\uly)=\begin{pmatrix} 1 & 0 \\ x_{m+1}  & 1 \end{pmatrix}\times C_{2m}(\ulx,\uly)
= \begin{pmatrix} F_{2m+1}(\uly,S\ulx) &  F_{2m}(\uly,S\ulx)\\ F_{2m+2}(\ulx,\uly) & F_{2m+1}(\ulx,\uly)  \end{pmatrix}$.
\end{enumerate}
\end{lemma}

\subsection{The quadratic dual}\label{ss:quadratic-dual}

Fix $n\geq 0$ and two sequences of positive integers $\ulx,\uly$. Let $A=A_n(\ulx,\uly)$.

The quadratic dual $A^!$ of $A$ is given by the same quiver as $A$ with relations $\beta_{i',j'}\alpha_{i,j}$ ($1\leq i'<i\leq \lfloor \frac{n+1}{2}\rfloor$, $1\leq j'\leq y_{i'}$ and $1\leq j\leq x_i$) and $\alpha_{i,j}\beta_{i',j'}$ ($1\leq i\leq i'\leq \lfloor\frac{n}{2}\rfloor$, $1\leq j\leq x_i$ and $1\leq j'\leq y_{i'}$).
We remark that this is opposite to the usual definition but it is more convenient for our purpose.

The algebra $A^!$ is naturally graded by path length. One checks that the degree $p$ component $(A^!)^p$ vanishes for $p>n$; the degree $n$ component $(A^!)^n$ has a basis
\[\{\beta_{m,j'_m}\alpha_{m,j_m}\cdots\beta_{2,j'_2}\alpha_{2,j_2}\beta_{1,j'_1}\alpha_{1,j_1}~|~1\leq j'_{i}\leq y_{i} \text{ and } 1\leq j_i\leq x_i \}\] for $n=2m$
and
\[\{\alpha_{m+1,j_{m+1}}\beta_{m,j'_m}\alpha_{m,j_m}\cdots\beta_{2,j'_2}\alpha_{2,j_2}\beta_{1,j'_1}\alpha_{1,j_1}~|~1\leq j'_{i}\leq y_{i}  \text{ and } 1\leq j_i\leq x_i \}\]
for $n=2m+1$; the degree $n-1$ component $(A^!)^{n-1}$ has a basis
\begin{align}
\{\alpha_{m,j_m}\cdots\beta_{2,j'_2}\alpha_{2,j_2}\beta_{1,j'_1}\alpha_{1,j_1}~|~1\leq j'_{i}\leq y_{i} \text{ and } 1\leq j_i\leq x_i \}\nonumber\\
\cup \{\beta_{m,j'_m}\alpha_{m,j_m}\cdots\beta_{2,j'_2}\alpha_{2,j_2}\beta_{1,j'_1}~|~1\leq j'_{i}\leq y_{i} \text{ and } 1\leq j_i\leq x_i \}\nonumber
\end{align} for $n=2m$
and
\begin{align}
\{\beta_{m,j'_m}\alpha_{m,j_m}\cdots\beta_{2,j'_2}\alpha_{2,j_2}\beta_{1,j'_1}\alpha_{1,j_1}~|~1\leq j'_{i}\leq y_{i}  \text{ and } 1\leq j_i\leq x_i \}\nonumber\\
\cup \{\alpha_{m+1,j_{m+1}}\beta_{m,j'_m}\alpha_{m,j_m}\cdots\beta_{2,j'_2}\alpha_{2,j_2}\beta_{1,j'_1}~|~1\leq j'_{i}\leq y_{i}  \text{ and } 1\leq j_i\leq x_i \}\nonumber
\end{align}
for $n=2m+1$.

By definition, the algebra $A$ is a quadratic monomial algebra, and hence Koszul by \cite[Proposition 2.2]{GreenZacharia94}.  Therefore, by \cite[Theorem 2.10.1]{BeilinsonGinzburgSoergel96} (Reminder: we are considering right modules), there is a graded algebra isomorphism between the opposite algebra of $A^!$ and the Yoneda algebra $\bigoplus_{p\geq 0}\Ext^p_{A}(S_1\oplus S_2,S_1\oplus S_2)$, where $S_1,S_2$ are the two simple modules of $A$. It follows that the global dimension of $A$ equals $n$.

\subsection{The top Hochschild cohomology}\label{ss:top-hoch-coh}

Fix $n\geq 0$ and two sequences of positive integers $\ulx,\uly$. Let $A=A_n(\ulx,\uly)$.
Since the global dimension of $A$ is $n$, it follows that $HH^p(A)$ vanishes for $p>n$ (see Section~\ref{ss:Hochschild-cohomologies}). In this subsection, we show that $HH^n(A)\neq 0$. For the case $\ulx=\uly=(1,1,1,\ldots)$, the Hochschild cohomologies of $A$ together with the algebra structure of $\bigoplus_{p\in\mathbb{Z}}HH^p(A)$ were determined by Fan and Xu in \cite{FanXu06,FanXu07}.

Recall that $A$ is Koszul. Then the Koszul bimodule complex provides a projective resolution of $A$ as a bimodule, by \cite[Section 3]{VandenBergh94}. Since we compute $HH^n$ only, it suffices to write down the components in degrees $-n$ and $-n+1$:
\[
\xymatrix{
A\otimes_{kQ_0}(A^!)^n\otimes_{kQ_0} A
\ar[r]^{f} & A\otimes_{kQ_0}(A^!)^{n-1}\otimes_{kQ_0} A}.
\]
% we can compute the following bimodule resolution
%$$0\ra A_n\otimes_{kQ_0}(A_n^!)^n\otimes_{kQ_0} A_n \xrightarrow{f} A_n\otimes_{kQ_0}(A_n^!)^{n-1}\otimes_{kQ_0} A_n \ra\ldots \ra A_n\otimes_{kQ_0}\otimes_{kQ_0}(A_n^!)^1\otimes_{kQ_0} A_n \ra A_n\otimes_{kQ_0}(A_n^!)^0\otimes_{kQ_0} A_n \xrightarrow{g} 0  $$
%is a projective resolution of $A_n$ as $A_n-A_n$-bimodule, where $g$ is given by
%$$A_n\otimes_{kQ_0}(A_n^!)^0 \otimes_{kQ_0} A_n \xrightarrow{\sim} A_n\otimes_{kQ_0} A_n \xrightarrow{mult} 0,$$
%and $f$ is given by $$1\otimes \rho_1\rho_2\ldots\rho_n \otimes 1 \mapsto \rho_1\otimes \rho_2\ldots\rho_n\otimes 1 - 1\otimes \rho_1\ldots\rho_{n-1}\otimes \rho_n.$$ Therefore the top Hochschild cohomology of $A_n$ has degree $n$, and is by definition the $n$-th selfextension of $A_n$ as $A_n-A_n$-bimodule, and thus the cokernel of $\Hom_{A_n^{op}\otimes A_n}(f,A_n)$.
Then $HH^n(A)$ is the cokernel of the homomorphism $f^*=\Hom_{A^e}(f,A)$. There are isomorphisms
\[
\xymatrix@R=0.4pc{
\Hom_{A^e}(A\otimes_{kQ_0}(A^!)^{n-1}\otimes_{kQ_0} A,A)
\ar[r]^{f^*}\ar[ddd]^{\cong} & \Hom_{A^e}(A\otimes_{kQ_0}(A^!)^{n}\otimes_{kQ_0} A,A)\ar[ddd]^{\cong}\\
\\
\\
\Hom_{A^e}((A^!)^{n-1}\otimes_{kQ_0^e} A^e,A)\ar[r]\ar[ddd]^\cong & \Hom_{A^e}((A^!)^{n}\otimes_{kQ_0^e} A^e,A)\ar[ddd]^\cong\\
\\
\\
\Hom_{kQ_0^e}((A^!)^{n-1},A)\ar[r] &\Hom_{kQ_0^e}((A^!)^{n},A)
%\\ \vspace{-10pt}\varphi\ar@{|->}[r] &(\rho_1\cdots\rho_n\mapsto %\varphi(\rho_1\cdots\rho_{n-1})\rho_n-\rho_1\varphi(\rho_2\cdots\rho_n))
}
\]
where the map on the last row is given by 
$$f^*(\varphi):\rho_1\cdots\rho_n\mapsto \varphi(\rho_1\cdots\rho_{n-1})\rho_n-\rho_1\varphi(\rho_2\cdots\rho_n)$$
for $\varphi:(A^!)^{n-1}\ra A$ and basis elements $\rho_1\cdots\rho_n\in(A^!)^n$.

%We describe this cokernel now more explicitly. Note that for all $i$,
%$$\Hom_{A_n^\op\otimes A_n}(A_n\otimes_{kQ_0}(A_n^!)^i\otimes_{kQ_0}A_n,A_n)\simeq
%\Hom_{A_n^\op\otimes A_n}((A_n^!)^{i}\otimes_{kQ_0^\op\otimes kQ_0}(A_n^\op\otimes A_n),A_n)
%\simeq \Hom{kQ_0^\op\otimes kQ_0}((A_n^!)^i,A_n).$$
%Thus the map $\Hom_{A_n^{op}\otimes A_n}(f,A_n)$ is isomorphic to
%$$f^*: \Hom{kQ_0^\op\otimes kQ_0}((A_n^!)^{n-1},A_n)
%ra \Hom{kQ_0^\op\otimes kQ_0}((A_n^!)^n,A_n),$$
%mapping $\varphi:(A_n^!)^{n-1} \ra A_n$ to a morphism which associates a path $\rho_1\cdots\rho_n$ in $A_n^!$ of length $n$ to an element $$\varphi(\rho\cdots\rho_{n-1})\rho_n-\rho_1\varphi(\rho_2\cdots\rho_n)$$ in $A_n$.
%The quadratic dual $A_n^!$ has a unique path of length $n$, namely $\rho:=\cdots\beta_2\alpha_2\beta_1\alpha_1$. Hence $\Hom{kQ_0^\op\otimes kQ_0}((A_n^!)^n,A_n)$ counts the paths in $A_n$ with the same starting and the same ending as $\rho$.
\smallskip

\begin{proposition}\label{p:top-hoch-coh}
Assume $n\geq 2$. Then the dimension of the top Hochschild cohomology of $A$ is
given by
\[
\dim_k HH^{n}(A) =\begin{cases} F_{n-1}(\uly,S\ulx)\prod_{i=1}^{\frac{n}{2}} x_iy_i & \text{ if $n$ is even}, \\
 F_{n+1}(\ulx,\uly)x_{\frac{n+1}{2}}\prod_{i=1}^{\frac{n-1}{2}} x_iy_i-2 & \text{ if $n$ is odd}.
 \end{cases}
\]
%\begin{eqnarray*}
%\dim_k HH^{2n}(A_{2n}(\ulx,\uly)) &=& F_{2n-1}(\uly,S\ulx)\\
%\dim_k HH^{2n+1}(A_{2n+1}(\ulx,\uly)) & = & F_{2n+2}(\ulx,\uly)-2.
%\end{eqnarray*}
\end{proposition}

\begin{proof}
For $n=2m$: Recall from Section~\ref{ss:quadratic-dual} that $(A^!)^n$ has a basis
$
\{\beta_{m}\alpha_{m}\cdots\beta_{2}\alpha_{2}\beta_{1}\alpha_{1} \}
$
and $(A^!)^{n-1}$ has a basis
$
\{\alpha_{m}\cdots\beta_{2}\alpha_{2}\beta_{1}\alpha_{1} \}
\cup \{\beta_{m}\alpha_{m}\cdots\beta_{2}\alpha_{2}\beta_{1}\}
$,
where $\beta_i$ ranges over $\{\beta_{i,1},\ldots,\beta_{i,y_i}\}$ and $\alpha_i$ ranges over $\{\alpha_{i,1},\ldots,\alpha_{i,x_i}\}$ for all $1\leq i\leq m$. All these basis elements of $(A^!)^n$ are cycles at $1$. Therefore,
\[
\dim_k\Hom_{kQ_0^e}((A^!)^{n},A)=\dim_k(e_1Ae_1)^{\oplus\prod_{i=1}^m x_iy_i}=F_{n-1}(\uly,S\ulx)\prod_{i=1}^m x_iy_i.
\] Here $\prod_{i=1}^m x_iy_i$ is the dimension of $(A^!)^n$, and $\dim_k(e_1Ae_1)$ is the $(1,1)$-entry in the Cartan matrix, namely, $F_{n-1}(\uly,S\ulx)$ by Lemma~\ref{l:cartan}.
The path $\beta_m\alpha_m\cdots\beta_2\alpha_1\beta_1$ in $(A^!)^{n-1}$ starts from $2$ and ends at $1$, so a morphism $\varphi\in\Hom_{kQ_0^e}((A^!)^{n-1},A)$ maps it to a linear combination of paths in $A$ starting from $2$ and ending at $1$. It follows that $\varphi(\beta_m\alpha_m\cdots\beta_2\alpha_1\beta_1)\alpha_1=0$ for all $\varphi\in\Hom_{kQ_0^e}((A^!)^{n-1},A)$. Similarly, $\beta_m\varphi(\alpha_m\cdots\beta_2\alpha_2\beta_1\alpha_1)=0$ for all $\varphi\in\Hom_{kQ_0^e}((A^!)^{n-1},A)$. We have shown that $f^*: \Hom_{kQ_0^e}((A^!)^{n-1},A)
\ra \Hom_{kQ_0^e}((A^!)^{n},A)$ is trivial. Therefore,
 \[
 \dim_k HH^{n}(A)=\dim_k \Hom_{kQ_0^e}((A^!)^{n},A)=F_{n-1}(\uly,S\ulx)\prod_{i=1}^m x_iy_i,\]
as desired.

For $n=2m+1$: Recall from Section~\ref{ss:quadratic-dual} that $(A^!)^n$ has a basis
$
\{\alpha_{m+1}\beta_{m}\alpha_{m}\cdots\beta_{2}\alpha_{2}\beta_{1}\alpha_{1} \}
$
and $(A^!)^{n-1}$ has a basis
$
\{\beta_m\alpha_m\cdots\beta_{2}\alpha_{2}\beta_{1}\alpha_{1} \}
\cup \{\alpha_{m+1}\beta_{m}\alpha_{m}\cdots\beta_{2}\alpha_{2}\beta_1\}
$,
where $\beta_i$ ranges over $\{\beta_{i,1},\ldots,\beta_{i,y_i}\}$ for all $1\leq i\leq m$ and $\alpha_i$ ranges over $\{\alpha_{i,1},\ldots,\alpha_{i,x_i}\}$ for all $1\leq i\leq m+1$. Therefore,
\[
\dim_k\Hom_{kQ_0^e}((A^!)^{n},A)=\dim_k(e_2Ae_1)^{\oplus x_{m+1}\prod_{i=1}^m x_iy_i}=F_{n+1}(\ulx,\uly)x_{m+1}\prod_{i=1}^m x_iy_i.\]
The path $\alpha_{m+1}\beta_m\alpha_m\cdots\beta_2\alpha_2\beta_1$ of $(A^!)^{n-1}$ is a cycle at $2$. A morphism $\varphi\in\Hom_{kQ_0^e}(A^!)^{n-1} \ra A$ maps it to a linear combination of cycles of $A$ at $2$. All such cycles become trivial when composed with $\alpha_1$ on the right, except the trivial circle $e_2$. So
\[
\varphi(\alpha_{m+1}\beta_m\alpha_m\cdots\beta_2\alpha_1\beta_1)\alpha_1\neq 0\Longleftrightarrow\varphi(\alpha_{m+1}\beta_m\alpha_m\cdots\beta_2\alpha_1\beta_1)=e_2.
\]
Similarly,
\[
\alpha_{m+1}\varphi(\beta_m\alpha_m\cdots\beta_2\alpha_2\beta_1\alpha_1)\neq 0\Longleftrightarrow\varphi(\beta_m\alpha_m\cdots\beta_2\alpha_2\beta_1\alpha_1)=e_1.
\]
We have shown that the image of $f^*: \Hom_{kQ_0^e}((A^!)^{n-1},A)
\ra \Hom_{kQ_0^e}((A^!)^{n},A)$ has dimension $2$. Therefore
\[
\dim_k HH^{n}(A)=\dim_k\Hom_{kQ_0^e}((A^!)^{n},A)-2=F_{n+1}(\ulx,\uly)x_{m+1}\prod_{i=1}^m x_iy_i-2,
\]
as desired.
\end{proof}

Recall that $A_1(\ulx,\uly)=B(x_1,0)$ and $A_0(\ulx,\uly)=B(0,0)$. Thus by Corollary~\ref{c:hochschild-qh}, we have $\dim_k HH^1(A_1(\ulx,\uly))=x_1^2-1$,  $HH^0(A_1(\ulx,\uly))=k$ and $HH^0(A_0(\ulx,\uly))=k\oplus k$.

\subsection{Derived simplicity}

By combining the results in the previous sections we obtain the derived simplicity of the algebras $A_n(\ulx,\uly)$ for all $n\geq 3$.

\begin{corollary}
Let $\ulx,\uly,\ulx',\uly'$ be four sequences of positive integers and $n,n'$ two non-negative integers. If $n\neq n'$, then the algebras $A_n(\ulx,\uly)$ and $A_{n'}(\ulx',\uly')$ are not derived equivalent. Furthermore, when $n\geq 3$, the algebras $A_n(\ulx,\uly)$ are all derived simple.
\end{corollary}

\begin{proof} 

To prove the first statement, let us suppose $n>n'$. If $n\geq 2$, then $HH^n(A_n(\ulx,\uly))\neq 0$ by Proposition~\ref{p:top-hoch-coh}, while $HH^n(A_{n'}(\ulx',\uly'))=0$ because $n>n'=\mathrm{gl.dim}  A_{n'}(\ulx',\uly')$. So $A_n(\ulx,\uly)$ and $A_{n'}(\ulx',\uly')$ are not derived equivalent. If $n=1$ and $n'=0$, then $A_{n'}(\ulx',\uly')$ is semisimple, while $A_n(\ulx,\uly)$ is not semisimple, so they are not derived equivalent.

In Section~\ref{ss:quadratic-dual} we have already shown that $A_n(\underline{x},\underline{y})$ has finite global dimension, whence it is derived simple or derived equivalent to a quasi-hereditary algebra by Theorem~\ref{t:dichotomy}. Thus the second statement is an immediate consequence of the first one, since up to Morita equivalence all quasi-hereditary algebras with two simple modules belong to the set $\{A_n(\ulx,\uly)~|~n\leq 2,~\ulx,\uly \text{ arbitrary}\}$.
\end{proof}

\subsection{Further properties}
In this subsection, we list some further properties of $A_n(\ulx,\uly)$ and $A_n(\ulx,\uly)^!$ and leave the proof to the interested reader.
We denote by $S_1$, $S_2$ and $P_1$, $P_2$ the simple and indecomposable projective modules corresponding to the vertices $1$ and $2$, respectively.

\begin{lemma} Fix two sequences of positive integers $\ulx,\uly$.
\begin{enumerate}
\item The socle of $A_{2m}(\ulx,\uly)$ is a direct sum of copies of $S_2$;
the socle of $A_{2m+1}(\ulx,\uly)$ is a direct sum of copies of $S_1$.

\item The indecomposable projective modules of $A_{2m}(\ulx,\uly)$ are obtained from those of $A_{2m-1}(\ulx,\uly)$, by adding $y_m$ copies of $S_2$ to each subfactor $S_1$.

\item The indecomposable projective modules of $A_{2m+1}(\ulx,\uly)$ are obtained from those of $A_{2m}(\ulx,\uly)$, by adding $x_{m+1}$ copies of $S_1$ to each subfactor $S_2$.

%\item The dimension of $A_n(\ulx,\uly)$ is $F_{n-1}(\uly,S\ulx)+F_{n}(\uly,S\ulx)+F_{n}(\ulx,\uly)+F_{n+1}(\ulx,\uly)$. $\times$

%\item The Loewy length of $A_n(\ulx,\uly)$ is $n$.

\item The socle of $A_n(\ulx,\uly)^!$ is a direct sum of copies of the simple module $S_1$.

\item The module $P_2$ over $A_{2m}(\ulx,\uly)^!$ is the same as over  $A_{2m-1}(\ulx,\uly)^!$; the module $P_1$ over $A_{2m}(\ulx,\uly)^!$ is obtained from the $P_1$ over $A_{2m-1}(\ulx,\uly)^!$ by adding $y_{m}$ copies of $P_2$ to the top of $P_1$.

\item The module $P_1$ over $A_{2m+1}^!(\ulx,\uly)$ is the same as $P_1$ over $A_{2m}^!(\ulx,\uly)$; the module $P_2$ over $A_{2m+1}^!(\ulx,\uly)$ is obtained from the $P_2$ over $A_{2m}^!(\ulx,\uly)$ by adding $x_{m+1}$ copies of $P_1$ to the top of $P_2$.

%\item The dimension of $A_n(\ulx,\uly)^!$ is $F_{n+3}(\ulx,\uly)$. $\times$

%\item The Loewy length of $A_n(\ulx,\uly)^!$ is $n$.

\end{enumerate}
\end{lemma}

%%%%
%\bibliographystyle{amsplain}
%\bibliography{stanYang}

\begin{thebibliography}{10}

\bibitem{AiharaIyama12}
Takuma Aihara and Osamu Iyama, \emph{Silting mutation in triangulated
  categories}, J. Lond. Math. Soc. (2) \textbf{85} (2012), no.~3, 633--668.

\bibitem{AngeleriKoenigLiuYang13}
Lidia {Angeleri H{\"u}gel}, Steffen Koenig, Qunhua Liu, and Dong Yang,
  \emph{Derived simple algebras and restrictions of recollements of derived
  module categories}, preprint (2013), arXiv:1310.3479.

\bibitem{BeilinsonGinzburgSoergel96}
Alexander Beilinson, Victor Ginzburg, and Wolfgang Soergel, \emph{Koszul
  duality patterns in representation theory}, J. Amer. Math. Soc. \textbf{9}
  (1996), no.~2, 473--527.

\bibitem{BeilinsonBernsteinDeligne82}
Alexander~A. Beilinson, Joseph Bernstein, and Pierre Deligne, \emph{Faisceaux
  pervers}, Ast{\'e}risque, vol. 100, Soc. Math. France, 1982 (French).

\bibitem{ClineParshallScott88b}
Edward Cline, Brian Parshall, and Leonard~L. Scott, \emph{Algebraic
  stratification in representation categories}, J. Algebra \textbf{117} (1988),
  no.~2, 504--521.

\bibitem{Cline-Parshall-Scott88}
\bysame, \emph{Finite-dimensional algebras and highest weight categories}, J.
  reine ang. Math. \textbf{391} (1988), 85--99.


\bibitem{DlabRingel89a}
Vlastimil Dlab and Claus~Michael Ringel, \emph{Quasi-hereditary algebras},
  Illinois J. Math. \textbf{33} (1989), 280--291.

\bibitem{FanXu06}
Jinmei Fan and Yunge Xu, \emph{On {H}ochschild cohomology rings of {F}ibonacci
  algebras}, Front. Math. China \textbf{1} (2006), no.~4, 526--537.

\bibitem{FanXu07}
\bysame, \emph{Hochschild cohomology groups of {F}ibonacci algebras}, Chin.
  Ann. Math. Ser. A \textbf{28} (2007), no.~3 (Chinese).

\bibitem{Green80}
Edward~L. Green, \emph{Remarks on projective resolutions}, Representation
  theory, II (Proc. Second Internat. Conf., Carleton Univ., Ottawa, Ont.,
  1979), Lecture Notes in Math., vol. 832, Springer, Berlin, 1980,
  pp.~259--279.

\bibitem{GreenZacharia94}
Edward~L. Green and Dan Zacharia, \emph{The cohomology ring of a monomial
  algebra}, Manuscripta Math. \textbf{85} (1994), 11--23.

\bibitem{Happel89}
Dieter Happel, \emph{Hochschild cohomology of finite-dimensional algebras},
  S{\'e}minaire d'Alg{\`e}bre Paul Dubreil et Marie-Paul Malliavin, 39{\`e}me
  Ann{\'e}e (Paris, 1987/1988), Lecture Notes in Math., vol. 1404, Springer,
  Berlin, 1989, pp.~108--126.

\bibitem{Happel91b}
\bysame, \emph{A family of algebras with two simple modules and {F}ibonacci
  numbers}, Arch. Math. (Basel) \textbf{57} (1991), no.~2, 133--139.

\bibitem{Keller03}
Bernhard Keller, \emph{Derived invariance of higher structures on the
  {H}ochschild complex}, preprint, 2003.

\bibitem{Keller94}
\bysame, \emph{Deriving {D}{G} categories}, Ann. Sci. {\'E}cole Norm. Sup. (4)
  \textbf{27} (1994), no.~1, 63--102.

\bibitem{Keller06d}
\bysame, \emph{On differential graded categories}, International Congress of
  Mathematicians. Vol. II, Eur. Math. Soc., Z{\"u}rich, 2006, pp.~151--190.

\bibitem{KellerYang11}
Bernhard Keller and Dong Yang, \emph{Derived equivalences from mutations of
  quivers with potential}, Adv. Math. \textbf{226} (2011), no.~3, 2118--2168.

\bibitem{Membrillo-Hernandez94}
Fausto~H. Membrillo-Hern\'andez, \emph{Quasi-hereditary algebras with two
  simple modules and {F}ibonacci numbers}, Commun. Algebra \textbf{22} (1994),
  no.~11, 4499--4509.
  
\bibitem{Ringel91}
Claus Michael Ringel, \emph{The category of modules with good filtration over a quasi-hereditary algebra has almost split sequences}, Math. Z. \textbf{208} (1991), 209--223.


\bibitem{Scott87}
Leonard~L. Scott, \emph{Simulating algebraic geometry with algebra. {I}. {T}he
  algebraic theory of derived categories}, The Arcata Conference on
  Representations of Finite Groups (Arcata, Calif., 1986), Proc. Sympos. Pure
  Math., vol. 47, Amer. Math. Soc., Providence, RI, 1987, pp.~271--281.
  
\bibitem{VandenBergh94}
Michel~Van den Bergh, \emph{Non-commutative homology of some three dimensional
  quantum spaces}, J. K-theory \textbf{8} (1994), 213--230.


\end{thebibliography}
%\end{document}

\def\cprime{$'$}
\providecommand{\bysame}{\leavevmode\hbox to3em{\hrulefill}\thinspace}
\providecommand{\MR}{\relax\ifhmode\unskip\space\fi MR }
\providecommand{\MRhref}[2]{%
  \href{http://www.ams.org/mathscinet-getitem?mr=#1}{#2}
}
\providecommand{\href}[2]{#2}

\end{document}